\documentclass[reqno]{amsart}
\usepackage{amsfonts}
\usepackage{amssymb}
\usepackage{mathrsfs}
\usepackage{times}
\usepackage{xcolor}
\usepackage[square,numbers,sort&compress]{natbib}

\usepackage{amsthm}

\newtheorem{innercustomgeneric}{\customgenericname}
\providecommand{\customgenericname}{}
\newcommand{\newcustomtheorem}[2]{%
  \newenvironment{#1}[1]
  {%
   \renewcommand\customgenericname{#2}%
   \renewcommand\theinnercustomgeneric{##1}%
   \innercustomgeneric
  }
  {\endinnercustomgeneric}
}

\newcustomtheorem{customthm}{Theorem}

\theoremstyle{plain}%
  \newtheorem{theorem}{Theorem}[section]
  \newtheorem{corollary}[theorem]{Corollary}
  \newtheorem{proposition}[theorem]{Proposition}
  \newtheorem{lemma}[theorem]{Lemma}
  \newtheorem{example}[theorem]{Example}
  
  \newtheorem{definition}[theorem]{Definition}

\newtheorem{remark}[theorem]{Remark}

\newfont{\hueca}{msbm10}
\def\hu #1{\hbox{\hueca #1}}\def\hu #1{\hbox{\hueca #1}}

\begin{document}

\title{Graded Lie-Rinehart algebras}

\thanks{The first author was  supported by the Centre for Mathematics of the University of Coimbra - UIDB/00324/2020, funded by the Portuguese Government through FCT/MCTES. Second and fourth authors are supported by the PCI of the UCA `Teor\'\i a de Lie y Teor\'\i a de Espacios de Banach' and by the PAI with project number FQM298. The second author also is supported by the 2014-2020 ERDF Operational Programme and by the Department of Economy, Knowledge, Business and University of the Regional Government of Andalusia FEDER-UCA18-107643. The third author was supported by Agencia Estatal de Investigaci\'on (Spain), grant MTM2016-79661-P (European FEDER support included, UE)}

\author[E. Barreiro]{Elisabete Barreiro}
\address{Elisabete~Barreiro,     University of Coimbra, CMUC, Department of Mathematics,  Apartado 3008,
3001-454 Coimbra, Portugal,   {\em E-mail}: {\tt mefb@mat.uc.pt}}{}

\author[A.J. Calder\'on]{A.J.~Calder\'on}
\address{Antonio J. Calder\'on,
Departamento de Matem\'aticas, Universidad de C\'adiz, Puerto Real, Espa\~na\\ {\em E-mail}: {\tt ajesus.calderon@uca.es}}{}

\author[Rosa M. Navarro]{Rosa M. Navarro}
\address{Rosa M. Navarro,
Departamento de Matem{\'a}ticas, Universidad de Extremadura, C{\'a}ceres, Espa\~na\\ {\em E-mail}: {\tt rnavarro@unex.es}}{}

\author[Jos\'{e} M. S\'{a}nchez]{Jos\'{e} M. S\'{a}nchez}
\address{Jos\'{e} M. S\'{a}nchez,
Departamento de Matem\'aticas, Universidad de C\'adiz, Puerto Real, Espa\~na\\ {\em E-mail}: {\tt txema.sanchez@uca.es}}{}

\begin{abstract}
We introduce the class of graded Lie-Rinehart algebras as a natural generalization of the one of graded Lie algebras. For $G$ an abelian group, we show that if $L$ is a tight $G$-graded Lie-Rinehart algebra over an associative and commutative $G$-graded algebra $A$ then $L$ and $A$ decompose as the orthogonal direct sums $L = \bigoplus_{i \in I}I_i$ and $A = \bigoplus_{j \in J}A_j$, where any $I_i$ is a non-zero ideal of $L$, any $A_j$ is a non-zero ideal of $A$, and both decompositions satisfy that for any $i \in I$ there exists a unique $j \in J$ such that $A_jI_i \neq 0$. Furthermore, any $I_i$ is a graded Lie-Rinehart algebra over $A_j$. Also, under mild conditions, it is shown that the above decompositions of $L$ and $A$ are by means of the family of their, respective, gr-simple ideals.

\medskip


\end{abstract}

\maketitle

%%%%%%%%%%%%%%%%%%%%%%%%%%%%%%%%%%%%%%%%%%%%%%%%%5
%%%%%%%%%%%%%%%%%%%%%%%%%%%%%%%%%%%%%%%%%%%%%%%%%5
\section{Introduction and first definitions}
%%%%%%%%%%%%%%%%%%%%%%%%%%%%%%%%%%%%%%%%%%%%%%%%%5
%%%%%%%%%%%%%%%%%%%%%%%%%%%%%%%%%%%%%%%%%%%%%%%%%5

Lie-Rinehart algebras were introduced by Herz in \cite{Herz}, being their theory mainly develo\-ped in \cite{Palais, Rinehart}. This notion can be regarded as a Lie $\mathbb F$-algebra which is simultaneosly an $A$-module, where $A$ is an associative and commutative $\mathbb F$-algebra, such that both structures are well-related. Huebschmann viewed this class of algebras as an algebraic counterpart of Lie algebroids defined over smooth manifolds, and his work has been develo\-ped through a series of articles (see \cite{Huebschmann, Huebschmann2, Huebschmann3}). In the last years, Lie-Rinehart algebras have been considered in many areas of Mathematics, parti\-culary from a geometric viewpoint (see for instance \cite{Mackenzie}) and of course from an algebraic viewpoint \cite{Recent1, Recent2, Recent3}. %Some generalizations of Lie-Rinehart algebras, such as Lie-Rinehart superalgebras \cite{Chemla} or restricted Lie-Rinehart algebras \cite{Dokas} have been recently studied.

In the present paper,  for $G$ an abelian grading group, we introduce the class of graded Lie-Rinehart algebras $(L,A)$ being $A$ an associative $G$-graded algebra. We study the structure of the aforementioned algebras which can be considered to be a natural extension of graded Lie algebras. Let us note that our techniques consist mainly  of  introducing notions of connections over the elements of $G$. Moreover, we will need to impose some restrictions, the first of these is the reasonably natural notion of ``tightness'' introduced along Section 4. Remark that for tight graded Lie-Rinehart algebras, we establish a decomposition of the pair $(L,A)$ into ideals labeled by certain equivalence classes.

Our paper is organized as follows. In Section 2 we develop connection techniques in the framework of Lie-Rinehart algebras $(L,A)$ and apply, as a first step, all of these techniques to the study of the inner structure of $L$. In Section 3 we get, as a second step, a decomposition of $A$ as direct sum of adequate ideals. In Section 4, we relate the results obtained in Section 2 and 3 on $L$ and $A$ to prove our above mentioned  main results. The final section is devoted to show that, under mild conditions, the given decompositions of $L$ and $A$ are by means of the family of their, corresponding, simple ideals.

\medskip

Firstly we recall the definition of  Lie-Rinehart algebra.
Let ${\mathbb F}$ be an arbitrary base field and $A$ a commutative and associative ${\mathbb F}$-algebra. A {\it derivation} on $A$ is a ${\mathbb F}$-linear map $D : A \to A$ which satisfies
\begin{equation}
D(ab) = D(a)b + aD(b) \hspace{0.2cm} \mbox{\it (Leibniz's law)}
\end{equation}
for all $a,b \in A$.
The set $\mbox{Der}(A)$ of all derivations of $A$ is a Lie ${\mathbb F}$-algebra with Lie bracket $[D,D'] = DD' - D'D$ and an $A$-module simultaneously. These two structures are related by the following identity $$[D,aD'] = a[D,D'] + D(a)D', \hspace{0.1cm} \mbox{for all} \hspace{0.1cm} a\in A \hspace{0.1cm} \mbox{and} \hspace{0.1cm} D,D' \in  \mbox{Der}(A).$$

\begin{definition}\rm
A {\it Lie-Rinehart algebra} over an associative and commutative ${\mathbb F}$-algebra $A$ (with product denoted by juxtaposition) is a Lie ${\mathbb F}$-algebra $L$ (with product $[\cdot,\cdot]$) endowed with an $A$-module structure and with a map (called {\it anchor}) $$\rho : L \to \mbox{Der}(A),$$ which is simultaneously an $A$-module and a Lie algebra homomorphism, and  such that the following relation holds
\begin{equation}\label{fundamental}
[v,aw] = a[v,w] + \rho(v)(a)w,
\end{equation}
for any $v,w \in L$ and $a \in A.$ We denote it by $(L,A)$ or just by $L$ if there is no possible confusion.
\end{definition}

\begin{example}\label{exa1}\rm
Any Lie algebra $L$ is a Lie-Rinehart algebra over $A := \mathbb{F}$ as a consequence of $\mbox{Der}({\mathbb F}) = 0$.
\end{example}

\begin{example}\label{exa2}\rm
Any associative and commutative $\mathbb F$-algebra $A$ gives rise to a Lie-Rinehart algebra by taking $L := \mbox{Der}(A)$ and $\rho := Id_{\mbox{Der}(A)}$.
\end{example}

A {\it subalgebra} $(S,A)$ of $(L,A),$ $S$ for short, is a Lie subalgebra of $L$ such that  $AS \subset S$ and satisfying that $S$ acts on $A$ via the composition $$ S\hookrightarrow L \overset{\rho} \to {\mbox{Der}(A)}.$$

\noindent A subalgebra $(I,A)$ of $(L,A)$, $I$ for short, is called an {\it ideal} if $I$ is a Lie ideal of $L$ and satisfies
\begin{equation}\label{cond_ideal} \rho(I)(A)L \subset I.
\end{equation}
As example of an ideal we have ${\rm Ker}\rho$, the kernel of $\rho$. We say that a Lie-Rinehart algebra $(L,A)$ is {\it simple} if $[L,L] \neq 0$, $AA \neq 0$, $AL \neq 0$ and its only ideals are $\{0\}$, $L$ and ${\rm Ker}\rho$.

Let us introduce the class of graded algebras in the framework of Lie-Rinehart algebras. We begin by recalling the definition of a graded algebra.

\begin{definition}\rm
Let $G$ be an abelian  multiplicative group with neutral element $1$. An algebra $A$ over an arbitrary base field $\mathbb{F}$ is {\em $G$-graded} or just {\em graded} if $A = \oplus_{g \in G}A_g$ satisfying $A_g  A_{g'}\subset A_{gg'}$, for $g,g'\in G$, where we are denoting by juxtaposition the respective products on $A$ and $G$.
\end{definition}

\begin{definition}\label{DEF 1.5}\rm
Let $G$ be an abelian grading group with product denoted by juxtaposition. We say that $(L,A)$ is a   {\it $G$-graded Lie-Rinehart algebra}, or just {\it graded Lie-Rinehart algebra}, if $L$ is a $G$-graded Lie ${\mathbb F}$-algebra and $A$ is a $G$-graded (associative and commutative) ${\mathbb F}$-algebra satisfying
\begin{eqnarray}
&& A_hL_g \subset L_{hg}, \label{EQ1}\\
&& \rho(L_g)(A_h) \subset A_{gh}, \label{EQ2}
\end{eqnarray}
for any $g,h\in G$.
\end{definition}

\noindent Split Lie algebras, graded Lie algebras and split Lie-Rinehart algebras are examples of graded Lie-Rinehart algebras. So the present paper extends the results obtained in \cite{YoLie,YoGradLie,YoSplitLieRinehart}.

We denote the {\em $G$-supports} of the grading in $L$ and in $A$ to the sets $$\Sigma_G := \{g \in G \setminus \{1\} : L_g \neq 0\} \hspace{1cm} \Lambda_G := \{h \in G \setminus \{1\} : A_h \neq 0\}$$  respectively, so we can write $$L = L_1 \oplus \Bigl(\bigoplus_{g \in \Sigma_G}L_g \Bigr) \hspace{1cm} A = A_1 \oplus \Bigl(\bigoplus_{h \in \Lambda_G}A_h \Bigr).$$

\medskip

Throughout this paper $(L,A)$ is a $G$-graded Lie-Rinehart algebra with restrictions neither on the dimension of $L,A$ nor the abelian group $G$, nor on the base field ${\hu F}$.

%%%%%%%%%%%%%%%%%%%%%%%%%%%%%%%%%%%%%%%%%%%%%%%%%5
%%%%%%%%%%%%%%%%%%%%%%%%%%%%%%%%%%%%%%%%%%%%%%%%%5
\section{Connections in $\Sigma_G$. Decompositions}
%%%%%%%%%%%%%%%%%%%%%%%%%%%%%%%%%%%%%%%%%%%%%%%%%5
%%%%%%%%%%%%%%%%%%%%%%%%%%%%%%%%%%%%%%%%%%%%%%%%%5

Let $G$ be an abelian grading group and $(L,A)$  a $G$-graded Lie-Rinehart algebra.
We define $\Sigma_G^{-1} := \{g^{-1} : g \in \Sigma_G\}$. In a similar way we define $\Lambda_G^{-1} := \{g^{-1} : g \in \Lambda_G\}.$ Finally, let us denote $$\hbox{ $\Sigma_G^{\pm} := \Sigma_G \cup \Sigma_G^{-1}$ \hspace{0.1cm} and \hspace{0.1cm} $\Lambda_G^{\pm} := \Lambda_G \cup \Lambda_G^{-1}.$}$$

\begin{definition}\label{connection}\rm
Let $g,g' \in \Sigma_G$. We say that $g$ is {\it $\Sigma_G$-connected} to $g'$ if there exists $\{g_1,g_2,\dots,g_n\}\subset \Sigma_G^{\pm} \cup \Lambda_G^{\pm}$ such that

\begin{enumerate}
\item[i.] $g_1 = g$.
\item[ii.] $\{g_1,g_1g_2,\dots,g_1g_2\cdots g_{n-1}\} \subset \Sigma_G^{\pm}.$
\item[iii.] $g_1g_2 \cdots g_n \in \{g',(g')^{-1}\}.$
\end{enumerate}

\noindent We also say that $\{g_1,\dots,g_n\}$ is a {\it $\Sigma_G$-connection} from $g$ to $g'$.
\end{definition}

The next result shows the $\Sigma_G$-connection relation is of equivalence. Its proof is virtually identical to the one of \cite[Proposition 2.1]{YoGradLie}, but we add a sketch of the same.

\begin{proposition}\label{pro1}
The relation $\sim$ in $\Sigma_G$, defined by $g \sim g'$ if and only if $g$ is $\Sigma_G$-connected to $g$, is of equivalence.
\end{proposition}
\begin{proof}
$\{g\}$ is a $\Sigma_G$-connection from $g$ to itself and therefore $g \sim g$.

If $g \sim g'$ and $\{g_1,\dots, g_n\}$ is a $\Sigma_G$-connection from $g$ to $g'$, then $$\{g_1 \cdots g_n, g_n^{-1}, g_{n-1}^{-1},
\dots, g_2^{-1}\} \subset \Sigma_G^{\pm} \cup \Lambda_G^{\pm}$$ is a $\Sigma_G$-connection from $g'$ to $g$ in case $g_1 \cdots g_n = g'$, and $$\{g_1^{-1} \cdots g_n^{-1}, g_n, g_{n-1},\dots,g_2\} \subset \Sigma_G^{\pm} \cup \Lambda_G^{\pm}$$ in case $g_1 \cdots g_n = (g')^{-1}$.
Therefore $g' \sim g$.

Finally, suppose $g \sim g'$ and $g' \sim g''$, and write
$\{g_1,\dots, g_n\}$ for a $\Sigma_G$-connection from $g$ to $g'$ and $\{{g'}_1,\dots, {g'}_m\}$ for a $\Sigma_G$-connection from
$g'$ to $g''$. If $m > 1$, then
$\{g_1,\dots,g_n,{g'}_2,\dots,{g'}_m\}$ is a $\Sigma_G$-connection from $g$ to $g''$ in case $g_1 \cdots g_n = g'$, and $\{g_1,\dots,g_n,{g'}_2^{-1},\dots,{g'}_m^{-1}\}$ in case $g_1 \cdots g_n = (g')^{-1}$. If $m = 1$, then $g'' \in \{g', (g')^{-1}\}$ and so $\{g_1,\dots, g_n\}$ is a $\Sigma_G$-connection from $g$ to $g''$. Therefore $g \sim g''$ and $\sim$ is of equivalence.
\end{proof}

\noindent By Proposition \ref{pro1} the $\Sigma_G$-connection relation defined in $\Sigma_G$ is of equivalence. From here, we can consider the quotient set $$\Sigma_G / \sim := \{[g]: g \in \Sigma_G\},$$ becoming $[g]$ the set of elements in $\Sigma_G$ which are $\Sigma_G$-connected to $g$. Our next goal is to associate an (adequate) ideal $I_{[g]}$ of the Lie-Rinehart algebra $(L,A)$ to each $[g]$.

\begin{lemma}\label{C_g cerrado}
If $g' \in [g] $ and $g'', g'g'' \in \Sigma_G$, then $g'', g'g'' \in [g]$.
\end{lemma}
\begin{proof}
Analogue to the proof of \cite[Lemma 2.1]{YoGradLie}.
\end{proof}

For $[g]$, with $g \in \Sigma_G$, we define $$L_{[g],1} := \Bigl(\sum\limits_{g' \in [g] \cap \Lambda_G}A_{(g')^{-1}}L_{g'} + \sum\limits_{g' \in [g]}[L_{(g')^{-1}}, L_{g'}]\Bigr) \subset L_1$$ and $$V_{[g]} := \bigoplus\limits_{g' \in [g]}L_{g'}.$$ Then we consider the following (graded) subspace of $L$, $$I_{[g]} := L_{[g],1} \oplus V_{[g]}.$$

\begin{proposition}\label{pro-2}
For any $[g] \in \Sigma_G/\sim$, the following assertions hold.
\begin{enumerate}
\item[i)] $[I_{[g]}, I_{[g]}] \subset I_{[g]}$.
\item[ii)] $AI_{[g]} \subset I_{[g]}.$
\end{enumerate}
\end{proposition}

\begin{proof}
i) We have
\begin{align}\label{cero1}
[I_{[g]}, I_{[g]}] &= [L_{[g],1} \oplus V_{[g]}, L_{[g],1} \oplus V_{[g]}]\nonumber\\
& \subset [L_{[g],1},L_{[g],1}] + [L_{[g],1},V_{[g]}]+ [V_{[g]},L_{[g],1}] + [V_{[g]}, V_{[g]}].
\end{align}

Let us consider the second summand in (\ref{cero1}). If there exist $g' \in [g]$ such that $A_{(g')^{-1}}L_{g'}$ is non-zero we have by Equation \eqref{EQ1} that $A_{(g')^{-1}}L_{g'} \subset L_{(g')^{-1}g'} \subset L_1$. If some $g' \in [g]$ satisfies $[L_{(g')^{-1}}, L_{g'}] \neq 0,$ we get as in the previous case $[L_{(g')^{-1}}, L_{g'}] \subset L_1$.
We have $[L_{[g],1}, L_{g'}] \subset L_{g'} \subset V_{[g]}$. In a similar way $[V_{[g]},L_{[g],1}] \subset V_{[g]}$ and we conclude for the second and third summand in Equation (\ref{cero1}) that
$$[L_{[g],1},V_{[g]}]+ [V_{[g]},L_{[g],1}] \subset V_{[g]}.$$
Consider now the fourth summand in (\ref{cero1}). Given $g', g'' \in [g]$ we get $[L_{g'}, L_{g''}] \subset L_{g'g''}.$ If $g'g'' = 1$ we have $[L_{g'}, L_{g''}] \subset L_{[g],1}.$ Suppose $g'g''\in \Sigma_G,$ then by Lemma \ref{C_g cerrado} we have $[L_{g'}, L_{g''}] \subset L_{g'g''} \subset V_{[g]}$.

Finally, for $g',g'' \in [g]$ the first summand of (\ref{cero1}) is $$\Bigl[\sum\limits_{g' \in [g]\cap \Lambda_G}A_{(g')^{-1}}L_{g'} + \sum\limits_{g' \in [g]}[L_{g'},L_{(g')^{-1}}],  \sum\limits_{g'' \in [g] \cap \Lambda_G}A_{(g'')^{-1}}L_{g''} + \sum\limits_{g'' \in [g]}[L_{g''},L_{(g'')^{-1}}]\Bigr] \subset$$

$$\sum\limits_{g',g'' \in [g] \cap \Lambda_G} [A_{(g')^{-1}}L_{g'},A_{(g'')^{-1}}L_{g''}]  +      \sum\limits_{g' \in [g] \cap \Lambda_G, g'' \in [g]} \Bigl[A_{(g')^{-1}}L_{g'}, [L_{g''},L_{(g'')^{-1}}]\Bigr]$$

\begin{equation}\label{MasSumandos}
+\sum\limits_{g' \in [g], g'' \in [g] \cap \Lambda_G}\Bigl[[L_{g'},L_{(g')^{-1}}],A_{(g'')^{-1}}L_{g''}\Bigr]   +    \sum\limits_{g',g'' \in [g]}\Bigl[[L_{g'},L_{(g')^{-1}}], [L_{g''},L_{(g'')^{-1}}]\Bigr]
\end{equation}

For the first summand in \eqref{MasSumandos}, if there exist $g',g'' \in [g] \cap \Lambda_G$ such that $[A_{(g')^{-1}}L_{g'},A_{(g'')^{-1}}L_{g''}]\neq 0$, by Equation \eqref{fundamental} and Equation \eqref{EQ2} is
\begin{align*}
[L_{(g')^{-1}g'},A_{(g'')^{-1}}L_{g''}] &\subset A_{(g'')^{-1}}[L_{(g')^{-1}g'},L_{g''}] + \rho(L_{(g')^{-1}g'})(A_{(g'')^{-1}})L_{g''}\\
& \subset A_{(g'')^{-1}}L_{g''}  \subset L_{[g],1}
\end{align*}

If there exist $g' \in [g] \cap \Lambda_G, g'' \in [g]$ such that the second summand of \eqref{MasSumandos} is non-zero, by Equation \eqref{fundamental} and Equation \eqref{EQ2} we have
\begin{align*}
\bigl[A_{(g')^{-1}}L_{g'},[L_{g''},L_{(g'')^{-1}}]\bigr] &= A_{(g')^{-1}}\bigl[[L_{g''},L_{(g'')^{-1}}],L_{g'}\bigr] +\rho([L_{g''},L_{(g'')^{-1}}])(A_{(g')^{-1}})L_{g'}\\
& \subset A_{(g')^{-1}}L_{g'} \subset L_{[g],1}
\end{align*}
The proof for the third summand in \eqref{MasSumandos} is similar.

For the fourth summand in \eqref{MasSumandos}, taking now into account the bilinearity of the product and Jacobi identity we obtain
$$\sum\limits_{g',g'' \in [g]}\bigl[[L_{g'},L_{(g')^{-1}}], [L_{g''},L_{(g'')^{-1}}]\bigr]
\subset $$
$$ \sum\limits_{g',g'' \in [g]}\Bigl(\Bigl[L_{g'},\bigl[L_{(g')^{-1}}, [L_{g''},L_{(g'')^{-1}}]\bigr]\Bigr] + \Bigl[L_{(g')^{-1}},\bigl[L_{g'}, [L_{g''},L_{(g'')^{-1}}]\bigr]\Bigr]\Bigr) \subset $$
$$\sum\limits_{g' \in [g]} \bigl([ L_{g'}, L_{(g')^{-1}}]+ [L_{(g')^{-1}},L_{g'}]\Bigr) \subset \sum\limits_{g' \in [g]} [L_{g'},L_{(g')^{-1}}]=L_{[g],1}.$$  Therefore all summands in \eqref{MasSumandos} are contained in $L_{[g],1}$, and consequently all summands in \eqref{cero1} are included in $I_{[g]}$ as desired.

ii) Observe that $$AI_{[g]} = \Bigl(A_1 \oplus \bigl(\bigoplus_{k \in \Lambda_G}A_k\bigr)\Bigr) \Bigl(\Bigl(\sum\limits_{g' \in [g] \cap \Lambda_G}A_{(g')^{-1}}L_{g'} + \sum\limits_{g' \in [g]}[L_{(g')^{-1}}, L_{g'}]\Bigr) \oplus \bigoplus\limits_{g' \in [g]}L_{g'}\Bigr).$$ We have to consider six cases:

$\bullet$ For $g' \in [g] \cap \Lambda_G$, since $L$ is an $A$-module we get
\begin{equation}\label{EQnuevaa}
A_1(A_{(g')^{-1}}L_{g'}) = (A_1A_{(g')^{-1}})L_{g'} \subset A_{(g')^{-1}}L_{g'} \subset L_{[g],1}.
\end{equation}

$\bullet$ For $g' \in [g]$, by Equation \eqref{fundamental} we get $$A_1[L_{(g')^{-1}},L_{g'}] \subset [L_{(g')^{-1}},A_1L_{g'}] + \rho(L_{(g')^{-1}})(A_1)L_{g'}.$$ Since $A_1L_{g'} \subset L_{g'}$ we get $[L_{(g')^{-1}},A_1L_{g'}] \subset [L_{(g')^{-1}},L_{g'}]$. Also, taking into account Equation \eqref{EQ2} we obtain $\rho(L_{(g')^{-1}})(A_1) \subset A_{(g')^{-1}}$. If $A_{(g')^{-1}} \neq 0$ (otherwise is trivial), $(g')^{-1} \in \Lambda_G$ therefore $\rho(L_{(g')^{-1}})(A_1)L_{g'} \subset A_{(g')^{-1}}L_{g'}$ with $g' \in [g] \cap \Lambda_G$. From here,
\begin{equation}\label{2}
A_1[L_{(g')^{-1}},L_{g'}] \subset L_{[g],1}.
\end{equation}

$\bullet$ For $g' \in [g],$ from the action of $A$ over $L$ it follows
\begin{equation}\label{3}
A_1L_{g'} \subset L_{g'} \subset V_{[g]}.
\end{equation}

$\bullet$ For $k \in \Lambda_G, g' \in [g] \cap \Lambda_G$, using that $A$ is conmmutative and $L$ is an $A$-module,
\begin{align*}
A_k(A_{(g')^{-1}}L_{g'}) &= (A_kA_{(g')^{-1}})L_{g'} = A_{(g')^{-1}}(A_kL_{g'}) \subset A_{(g')^{-1}}L_{kg'} \subset L_k
\end{align*}
If $L_k \neq 0$ (otherwise is trivial), we have $k \in \Sigma_G$ and then with the $\Sigma_G$-connection $\{g',k,(g')^{-1}\}$ we have $k \in [g]$. That is,
\begin{equation}\label{EQnuevaaa}
L_k \subset V_{[g]}.
\end{equation}

$\bullet$ For $k \in \Lambda_G, g' \in [g]$ using Equation \eqref{fundamental} and Equation \eqref{EQ2} we obtain
\begin{align*}
A_k[L_{(g')^{-1}},L_{g'}] &\subset [L_{(g')^{-1}},A_kL_{g'}] + \rho(L_{(g')^{-1}})(A_k)L_{g'}\\
& \subset [L_{(g')^{-1}},L_{kg'}] + A_{k(g')^{-1}}L_{g'} \subset L_k
\end{align*}
As in the previous case, if $L_k \neq 0$ we get $k \in \Sigma_G$ and $k \in [g]$. That is,
\begin{equation}\label{5}
A_k[L_{(g')^{-1}},L_{g'}] \subset V_{[g]}.
\end{equation}

$\bullet$ For $k \in \Lambda_G, g' \in [g]$ we obtain $A_kL_{g'} \subset L_{kg'}$. If $kg' \in \Sigma_G,$ by using the $\Sigma_G$-connection $\{g',k\}$ we get $g' \sim kg',$ and by transitivity $kg' \in [g]$, meaning that
\begin{equation}\label{6}
A_kL_{g'} \subset V_{[g]}.
\end{equation}
From Equations \eqref{EQnuevaa}-\eqref{6}, assertion ii) is proved.
\end{proof}

\begin{proposition}\label{pro-9}
Let $[g],[h] \in \Sigma_G / \sim$ with $[g] \neq [h]$. Then $[I_{[g]}, I_{[h]}] = 0$.
\end{proposition}

\begin{proof}
We have
\begin{align}\label{cuatro1}
[I_{[g]}, I_{[h]}] &= [L_{[g],1} \oplus V_{[g]}, L_{[h],1} \oplus V_{[h]}] \nonumber\\
& \subset [L_{[g],1},L_{[h],1}] + [L_{[g],1} V_{[h]}] + [V_{[g]}, L_{[h],1}] +[V_{[g]}, V_{[h]}].
\end{align}

Consider the above fourth summand $[V_{[g]},V_{[h]}]$ and suppose there exist $g' \in [g]$ and $h' \in [h]$ such that $[L_{g'},L_{h'}]\neq 0$. As necessarily $g' \neq (h')^{-1}$, then $g'h' \in \Sigma_G$. Since $g \sim g'$ and $g'h' \in \Sigma_G,$ by Lemma \ref{C_g cerrado} we conclude $g \sim g'h'.$ Similarly we can prove $h \sim g'h',$ so we conclude $g \sim h,$ a contradiction. Hence $[L_{g'},L_{h'}] = 0$ and so
\begin{equation}\label{nueve1}
[V_{[g]},V_{[h]}]=0.
\end{equation}

For the second summand in \eqref{cuatro1} suppose there exist $g' \in [g] \cap \Lambda_G$ and $h' \in [h]$ such that $\bigl[A_{(g')^{-1}}L_{g'} + [L_{(g')^{-1}},L_{g'}],L_{h'}\bigr] \neq 0$. Then, $[A_{(g')^{-1}}L_{g'},L_{h'}]$ or
$\bigl[[L_{(g')^{-1}},L_{g'}],L_{h'}\bigr]$ is non-zero. In the first case, by Equation \eqref{fundamental} is
\begin{align*}
[A_{(g')^{-1}}L_{g'},L_{h'}] &= A_{(g')^{-1}}[L_{h'},L_{g'}] +\rho(L_{h'})(A_{(g')^{-1}})L_{g'}\\
& \subset A_{(g')^{-1}}L_{h'g'} +A_{h'(g')^{-1}}L_{g'}.
\end{align*}
If either $L_{h'g'}\neq 0$ or $A_{h'(g')^{-1}}\neq 0$ with the $\Sigma_G$-connections $\{g',h',(g')^{-1}\}$ or $\{g',h'(g')^{-1}\}$, respectively, we conclude $g' \sim h'$, that is, $[g]=[h]$, a contradiction. In the second case, by using Jacobi identity is $$\bigl[[L_{(g')^{-1}},L_{g'}],L_{h'}\bigr] = \bigl[[L_{g'},L_{h'}],L_{(g')^{-1}}\bigr] + \bigl[[L_{h'},L_{(g')^{-1}}],L_{g'}\bigr]$$ and by Equation \eqref{nueve1} that $$[L_{g'},L_{h'}] = [L_{h'},L_{(g')^{-1}}] = 0.$$ Hence  $\bigl[[L_{(g')^{-1}},L_{g'}],L_{h'}\bigr] = 0$ and we show
\begin{equation}\label{nueve2}
[L_{[g],1},V_{[h]}] = 0.
\end{equation}
Similarly can be proved for the third summand $[V_{[g]},L_{[h],1}] = 0.$

Finally, the first summand $[L_{[g],1},L_{[h],1}]$ in \eqref{cuatro1} is $$\Bigl[\sum\limits_{g' \in [g] \cap \Lambda_G}A_{(g')^{-1}}L_{g'} + \sum\limits_{g' \in [g]}[L_{(g')^{-1}}, L_{g'}], \sum\limits_{h' \in [h] \cap \Lambda_G}A_{(h')^{-1}}L_{h'} + \sum\limits_{h' \in [h]}[L_{(h')^{-1}}, L_{h'}]\Bigr] \subset$$

$$\sum\limits_{g'\in [g] \cap \Lambda_G,h' \in [h] \cap \Lambda_G} [A_{(g')^{-1}}L_{g'},A_{(h')^{-1}}L_{h'}] +      \sum\limits_{g' \in [g] \cap \Lambda_G, h'\in [h]} \Bigl[A_{(g')^{-1}}L_{g'}, [L_{(h')^{-1}}, L_{h'}]\Bigr]$$

\begin{equation}\label{MasSumandos2}
+\sum\limits_{g' \in [g],h'\in [h] \cap \Lambda_G}\Bigl[[L_{(g')^{-1}},L_{g'}],A_{(h')^{-1}}L_{h'}\Bigr]   +    \sum\limits_{g' \in [g], h'\in [h]}\Bigl[[L_{(g')^{-1}},L_{g'}], [L_{(h')^{-1}}, L_{h'}]\Bigr]
\end{equation}

For the first summand in \eqref{MasSumandos2}, if there exist $g'\in [g] \cap \Lambda_G,h' \in [h] \cap \Lambda_G$ such that $[A_{(g')^{-1}}L_{g'},A_{(h')^{-1}}L_{h'}]\neq 0$, by Equation \eqref{fundamental} and Equation \eqref{EQ2} is
{\small \begin{align*}
[A_{(g')^{-1}}L_{g'},A_{(h')^{-1}}L_{h'}] &\subset A_{(h')^{-1}}[A_{(g')^{-1}}L_{g'},L_{h'}] +\rho(A_{(g')^{-1}}L_{g'})(A_{(h')^{-1}})L_{h'}\\
&\subset A_{(h')^{-1}}\Bigl(A_{(g')^{-1}}[L_{h'},L_{g'}] + \rho(L_{h'})(A_{(g')^{-1}})L_{g'}\Bigr)\\
&+\rho(A_{(g')^{-1}}L_{g'})(A_{(h')^{-1}})L_{h'}\\
&\subset A_{(h')^{-1}}\Bigl(A_{(g')^{-1}}L_{h'g'} + A_{h'(g')^{-1}}L_{g'}\Bigr)+ A_{(g')^{-1}}A_{g'(h')^{-1}}L_{h'}\\
&\subset A_{(h')^{-1}}A_{(g')^{-1}}L_{h'g'} + A_{(h')^{-1}}A_{h'(g')^{-1}}L_{g'} + A_{(g')^{-1}}A_{g'(h')^{-1}}L_{h'}
\end{align*}}
If $A_{(h')^{-1}}A_{(g')^{-1}}L_{h'g'}$ is non-zero with the connection $\{g',h',(g')^{-1}\}$ we conclude $[g]=[h]$, a contradiction. Similarly, if $A_{(h')^{-1}}A_{h'(g')^{-1}}L_{g'}$ or $A_{(g')^{-1}}A_{g'(h')^{-1}}L_{h'}$ is non-zero with the $\Sigma_G$-connection $\{g',h'(g')^{-1}\}$ or $\{h',g'(h')^{-1}\}$, respectively, we obtain the same contradiction.

For $g'\in[g],h' \in [h] \cap \Lambda_G$ in the third summand of \eqref{MasSumandos2} using Equation \eqref{fundamental} we obtain two summands. For the first summand we use Jacobi identity and Equation \eqref{nueve1}, and for the second one we apply Equation \eqref{EQ2} and that $\rho$ is a Lie algebra homomorphism, we have
\begin{align*}
\Bigl[[L_{(g')^{-1}},L_{g'}],A_{(h')^{-1}}L_{h'}\Bigr] &= A_{(h')^{-1}}\Bigl[[L_{(g')^{-1}},L_{g'}],L_{h'}\Bigr] + \rho([L_{(g')^{-1}},L_{g'}])(A_{(h')^{-1}})L_{h'}\\
&\subset A_{(h')^{-1}}\Bigl( \bigl[[L_{(g')^{-1}},L_{h'}],L_{g'}\bigr] + \bigl[[L_{g'},L_{h'}],L_{(g')^{-1}}\bigr] \Bigr)\\
&\hspace{0.4cm} +\rho([L_{(g')^{-1}},L_{g'}])(A_{(h')^{-1}})L_{h'}\\
&\subset \bigl[\rho(L_{(g')^{-1}}),\rho(L_{g'})\bigr](A_{(h')^{-1}})L_{h'}\\
&\subset \Bigl(\rho(L_{(g')^{-1}}) \Bigl(\rho(L_{g'})(A_{(h')^{-1}})\Bigr)\Bigr)L_{h'}\\
&\hspace{0.4cm} + \Bigl(\rho(L_{g'})\Bigl(\rho(L_{(g')^{-1}})(A_{(h')^{-1}})\Bigr)\Bigr)L_{h'}\\
&\subset \Bigl(\rho(L_{(g')^{-1}}) (A_{g'(h')^{-1}})\Bigr)L_{h'} + \Bigl(\rho(L_{g'})(A_{(g')^{-1}(h')^{-1}})\Bigr)L_{h'}
\end{align*}
If $\bigl(\rho(L_{(g')^{-1}}) (A_{g'(h')^{-1}})\bigr)L_{h'}$ or $\bigl(\rho(L_{g'})(A_{(g')^{-1}(h')^{-1}})\bigr)L_{h'}$ is non-zero considering the $\Sigma_G$-connection $\{(g')^{-1}, g'(h')^{-1}\}$ or $\{g',(g')^{-1}(h')^{-1}\}$, respectively, we get $[g]=[h]$.

The proof for the second summand in \eqref{MasSumandos2} is analogous.

For the fourth summand in \eqref{MasSumandos2}, taking now into account the bilinearity of the product and Jacobi identity we obtain
$$\sum\limits_{g' \in [g],h' \in [h]}\bigl[[L_{(g')^{-1}},L_{g'}], [L_{(h')^{-1}},L_{h'}]\bigr] \subset $$
$$ \sum\limits_{g'\in [g],h' \in [h]}\Bigl(\Bigl[L_{(g')^{-1}},\bigl[L_{g'}, [L_{(h')^{-1}},L_{h'}]\bigr]\Bigr] + \Bigl[L_{g'},\bigl[L_{(g')^{-1}}, [L_{(h')^{-1}},L_{h'}]\bigr]\Bigr]\Bigr) \subset $$
$$\sum\limits_{g'\in [g]}\Bigl(\Bigl[L_{(g')^{-1}},\bigl[L_{g'}, L_{[h],1}\bigr]\Bigr] + \Bigl[L_{g'},\bigl[L_{(g')^{-1}}, L_{[h],1}\bigr]\Bigr]\Bigr)$$ and by Equation \eqref{nueve2} we also conclude this fourth summand is zero. That is, Equation \eqref{MasSumandos2} vanishes and then
\begin{equation}\label{nueve3}
[L_{[g],1},L_{[h],1}] =0.    
\end{equation}
In conclusion, from Equations \eqref{cuatro1}-\eqref{nueve3} is $[I_{[g]}, I_{[h]}] = 0.$
\end{proof}

We understand the usual regularity concepts in the graded sense. Given a Lie-Rinehart algebra $(L,A)$, an ideal $I$ of $L$ is a {\em graded ideal} if it splits as $I = \oplus_{g \in G}I_g$, where $I_g := I \cap L_g$, and satisfies $[I_g,I_{g'}] \subset I_{gg'}$, for $g,g'\in G$. We also say that $(L,A)$ is {\it gr-simple} if $[L,L] \neq 0$, $AA \neq 0$, $AL \neq 0$ and its only graded ideals are $\{0\}$, $L$ and ${\rm Ker}\rho$.

\begin{theorem}\label{teo-1}
The following assertions hold.
\begin{enumerate}
\item[{\rm i)}] For all $[g] \in \Sigma_G/ \sim$, the linear space $I_{[g]} = L_{[g],1} \oplus V_{[g]}$ is a graded ideal of $L$.

\item[{\rm ii)}] If $L$ is gr-simple then all the elements of $\Sigma_G$ are $\Sigma_G$-connected. Moreover, $$L_1 = \Bigl(\sum\limits_{g \in \Sigma_G \cap \Lambda_G}A_{g^{-1}}L_g\Bigr) + \Bigl(\sum\limits_{g \in \Sigma_G}[L_{g^{-1}}, L_g]\Bigr).$$
\end{enumerate}
\end{theorem}

\begin{proof}
i) We get $[I_{[g]},L_1] \subset I_{[g]}$ and by Propositions \ref{pro-2}-i) and \ref{pro-9} we have $$[I_{[g]}, L] = \Bigl[I_{[g]}, L_1 \oplus \bigl(\bigoplus\limits_{g' \in [g]}L_{g'}\bigr) \oplus \bigl(\bigoplus\limits_{h \notin [g]}L_h\bigr)\Bigr] \subset I_{[g]},$$ so $I_{[g]}$ is a Lie ideal of $L.$ Clearly by Proposition \ref{pro-2}-ii) we also have that  $I_{[g]}$ is an $A$-module. Finally, by Equation \eqref{fundamental} we have
$$\rho(I_{[g]})(A)L \subset [I_{[g]},AL] + A[I_{[g]},L] \subset [I_{[g]},L] + AI_{[g]} \subset I_{[g]}.$$
and we conclude $I_{[g]}$ is an ideal of $L.$ Since by construction $I_{[g]}$ is graded, we obtain the required result.

\medskip

ii) The gr-simplicity of $L$ implies $I_{[g]} \in \{{\rm Ker}\rho,L\}$ for any $g \in \Sigma_G.$ If some $g \in \Sigma_G$ is such that $I_{[g]} = L,$ then $[g] = \Sigma_G.$ Otherwise, if $I_{[g]} = {\rm Ker}\rho$ for all $g \in \Sigma_G$ is $[g] = [h]$ for any $g,h \in \Sigma_G$, and again $[g] = \Sigma_G.$ Therefore in any case $\Sigma_G$ has all its elements $\Sigma_G$-connected and $L_1 = \bigl(\sum_{g \in \Sigma_G \cap \Lambda_G}A_{g^{-1}}L_g\bigr) + \bigl(\sum_{g \in \Sigma_G}[L_{g^{-1}}, L_g]\bigr).$
\end{proof}

\begin{theorem}\label{teo-2}
Let $(L,A)$ be a graded Lie-Rinehart algebra. Then $$L = U + \sum\limits_{[g] \in \Sigma_G/\sim}I_{[g]},$$ where $U$ is a linear complement of $\bigl(\sum_{g \in \Sigma_G \cap \Lambda_G}A_{g^{-1}}L_g\bigr) + \bigl(\sum_{g \in \Sigma_G}[L_{g^{-1}}, L_g]\bigr)$ in $L_1$, and any $I_{[g]} \subset L$ is one of the graded ideals described in Theorem \ref{teo-1}-i). Furthermore, $[I_{[g]}, I_{[h]}] = 0$ when $[g] \neq [h].$
\end{theorem}

\begin{proof}
We have $I_{[g]}$ is well-defined and, by Theorem \ref{teo-1}-i), an ideal of $L$, being clear that $$L = L_1 \oplus \bigl(\bigoplus\limits_{g \in \Sigma_G}L_g\bigr) = U + \sum\limits_{[g] \in \Sigma_G/\sim}I_{[g]}.$$ Finally, Proposition \ref{pro-9} gives $[I_{[g]}, I_{[h]}]=0$ if $[g] \neq [h].$
\end{proof}

For a Lie-Rinehart algebra $L,$ we denote by ${\mathcal Z}(L) := \bigl\{v \in L : [v, L] = \rho(v) = 0 \bigr\}$ the {\it center} of $L$ as in reference \cite{center}.

\begin{corollary}\label{coro-1}
If ${\mathcal Z}(L) = 0$ and $L_1 = \bigl(\sum_{g \in \Sigma_G \cap \Lambda_G}A_{g^{-1}}L_g\bigr) + \bigl(\sum_{g \in \Sigma_G}[L_{g^{-1}}, L_g]\bigr)$ then $L$ is the direct sum of the graded ideals given in Theorem \ref{teo-1}-i), $$L = \bigoplus\limits_{[g] \in \Sigma_G/\sim}I_{[g]}.$$ Moreover, $[I_{[g]}, I_{[h]}] = 0$ when $[g] \neq [h].$
\end{corollary}

\begin{proof}
Since $L_1 = \bigl(\sum_{g \in \Sigma_G \cap \Lambda_G}A_{g^{-1}}L_g\bigr) + \bigl(\sum_{g \in \Sigma_G}[L_{g^{-1}}, L_g]\bigr)$ we get $$L = \sum\limits_{[g] \in \Sigma_G/\sim} I_{[g]}.$$
To verify the direct character of the sum, take some $v \in I_{[g]} \cap \bigl(\sum_{[h]\in \Sigma_G/\sim, [h] \neq [g]}I_{[h]}\bigr)$. Since $v \in I_{[g]},$ the fact $\bigl[I_{[g]},I_{[h]}\bigr] = 0$ when $[g] \neq [h]$ gives us $$\Bigl[v,\sum_{[h] \in \Sigma_G/ \sim, [h] \neq [g]}I_{[h]}\Bigr] = 0.$$

\noindent Similarly, since $v \in \sum_{[h]\in \Sigma_G/ \sim, [h] \neq [g]}I_{[h]}$ we get $[v,I_{[g]}] = 0.$ Therefore $[v,L]=0.$ Now, Equation \eqref{fundamental} allows us to conclude $\rho(v) = 0$. That is, $v \in {\mathcal Z}(L)=0$.
\end{proof}

%%%%%%%%%%%%%%%%%%%%%%%%%%%%%%%%%%%%%%%%%%%%%%%%%5
%%%%%%%%%%%%%%%%%%%%%%%%%%%%%%%%%%%%%%%%%%%%%%%%%5
\section{Connections in the $G$-support of $A$. Decompositions of $A$}
%%%%%%%%%%%%%%%%%%%%%%%%%%%%%%%%%%%%%%%%%%%%%%%%%5
%%%%%%%%%%%%%%%%%%%%%%%%%%%%%%%%%%%%%%%%%%%%%%%%%5

In this section we introduce an adequate notion of connection among the elements of the $G$-support $\Lambda_G$ for a commutative and associative $\mathbb{F}$-algebra $A$ associated with a $G$-graded Lie-Rinehart $\mathbb{F}$-algebra $L$ (see Definition \ref{DEF 1.5}). We recall that $A$ admits a group graduation as $$A=A_1 \oplus \Bigl(\bigoplus_{g \in \Lambda_G}A_g\Bigr),$$ being $\Lambda_G=\{g \in G\setminus \{1\}: A_g \neq 0\}$. As in the previous section, we will continue considering the sets $\Sigma_G^{\pm}$, $\Lambda_G^{\pm}$.

\begin{definition}\label{connection2}\rm
Let $g,g' \in \Lambda_G$. We say that $g$ is {\it $\Lambda_G$-connected} to $g'$ if there exists $\{k_1,k_2,\dots,k_n\} \subset \Lambda_G^{\pm} \cup \Sigma_G^{\pm}$ such that
\begin{enumerate}
\item[i.] $k_1 = g$.
\item[ii.] $\{k_1,k_1k_2,\dots,k_1k_2\cdots k_{n-1}\} \subset \Lambda_G^{\pm} \cup \Sigma_G^{\pm}.$
\item[iii.] $k_1k_2 \cdots k_n \in \{g',(g')^{-1}\}.$
\end{enumerate}

\noindent We also say that $\{k_1,\dots,k_n\}$ is a $\Lambda_G$-{\it connection} from $g$ to $g'$.
\end{definition}

\noindent As in Section 2 we can prove the next results.

\begin{proposition}\label{pro-1}
The relation $\approx$ in $\Lambda_G$, defined by $g \approx g'$ if and only if $g$ is $\Lambda_G$-connected to $g'$, is an equivalence relation.
\end{proposition}

\begin{remark}\label{Remark2}
Let $g,g' \in \Lambda_G$ such that $g \approx g'.$ If $g'h \in \Lambda_G,$ for $h \in \Lambda_G,$ then $g \approx g'h$. %Considering the connection $\{g',h\}$ we get $g' \approx g'h,$ and by transitivity $g \approx g'h$.
\end{remark}

\noindent By Proposition \ref{pro-1} we can consider the quotient set $$\Lambda_G/ \approx := \{[g]: g \in \Lambda_G\},$$ becoming $[g]$ the set of elements of $\Lambda_G$ which are $\Lambda_G$-connected to $g$. Our next goal in this section is to associate an (adequate) ideal $\mathcal{A}_{[g]}$ of the algebra $A$ to any $[g] \in \Lambda_G/ \approx$. Fix $g \in \Lambda_G$, we start by defining the sets $$A_{[g],1} := \Bigl(\sum_{g' \in [g] \cap  \Sigma_G}\rho(L_{(g')^{-1}})(A_{g'})\Bigr) + \Bigl(\sum_{g' \in [g]}A_{(g')^{-1}}A_{g'}\Bigr) \subset A_1$$ and $$A_{[g]} := \bigoplus \limits_{g' \in [g]} A_{g'}.$$
Hence, we denote by $\mathcal{A}_{[g]}$ the direct sum of the two graded subspaces above. That is, $$\mathcal{A}_{[g]} := A_{[g],1} \oplus A_{[g]}.$$

\begin{proposition}\label{pro2}
For any $[g] \in \Lambda_G/ \approx$ we have $\mathcal{A}_{[g]}\mathcal{A}_{[g]} \subset \mathcal{A}_{[g]}$.
\end{proposition}

\begin{proof}
Since the algebra $A$ is commutative we have
\begin{equation}\label{cero}
\mathcal{A}_{[g]}\mathcal{A}_{[g]} = \Bigl(A_{[g],1} \oplus A_{[g]}\Bigr)\Bigl(A_{[g],1} \oplus A_{[g]}\Bigl) \subset A_{[g],1}A_{[g],1} + A_{[g],1}A_{[g]} + A_{[g]}A_{[g]}.
\end{equation}

Let us consider the second summand in Equation (\ref{cero}). Given $g' \in [g]$ we have $A_{[g],1}A_{g'} \subset A_1A_{g'}\subset A_{g'}$, and therefore
\begin{equation}\label{ceroo}
A_{[g],1}A_{g'} \subset A_{[g]}.
\end{equation}

For the third summand in \eqref{cero}, let be $g', g'' \in [g]$ such that $0 \neq A_{g'}A_{g''} \subset A_{g'g''}.$ If $g'g'' = 1$ we have $A_{(g')^{-1}}A_{g'} \subset A_1,$ and so $A_{(g')^{-1}}A_{g'} \subset A_{[g],1}.$ Suppose $g'g'' \in \Lambda_G$, then by Remark \ref{Remark2} we have $g'g'' \in [g]$ and so $A_{g'}A_{g''} \subset A_{g'g''} \subset A_{[g]}$. Hence $A_{[g]}A_{[g]} = (\bigoplus_{g' \in [g]}A_{g'})(\bigoplus_{g'' \in [g]}A_{g''}) \subset A_{[g],1} \oplus A_{[g]}$. That is,
\begin{equation}\label{eq0.5}
A_{[g]}A_{[g]} \subset \mathcal{A}_{[g]}.
\end{equation}

Finally we consider the first summand $A_{[g],1}A_{[g],1}$ in \eqref{cero} and suppose there exist $g', g'' \in [g] \cap \Sigma_G$ such that $$\Bigl(\rho(L_{(g')^{-1}})(A_{g'}) + A_{(g')^{-1}}A_{g'}\Bigr)\Bigl(\rho(L_{(g'')^{-1}})(A_{g''}) + A_{(g'')^{-1}}A_{g''}\Bigr) \neq 0,$$ so
\begin{eqnarray}
&& \rho(L_{(g')^{-1}})(A_{g'})\rho(L_{(g'')^{-1}})(A_{g''}) + \rho(L_{(g')^{-1}})(A_{g'})(A_{(g'')^{-1}}A_{g''}) \nonumber \\
&& \quad \quad + (A_{(g')^{-1}}A_{g'})\rho(L_{(g'')^{-1}})(A_{g''})
+ (A_{(g')^{-1}}A_{g'})(A_{(g'')^{-1}}A_{g''}) \neq 0 \label{ideal_A}
\end{eqnarray}
For the last summand in Equation \eqref{ideal_A}, in case $g'' \neq (g')^{-1},$ by the commutativity and associativity of $A$ we have
$$(A_{(g')^{-1}}A_{g'})(A_{(g'')^{-1}}A_{g''}) = (A_{(g')^{-1}}A_{(g'')^{-1}})(A_{g'}A_{g''}) \subset A_{(g'g'')^{-1}}A_{g'g''}\subset A_{[g],1}$$
because by Remark \ref{Remark2} we get $g'g''\in [g]$. In case $g'' = (g')^{-1},$ it follows
$$(A_{(g')^{-1}}A_{g'})(A_{g'}A_{(g')^{-1}}) = A_{(g')^{-1}}(A_{g'}A_{g'}A_{(g')^{-1}}) \subset A_{(g')^{-1}}A_{g'} \subset A_{[g],1}.$$
For the second summand in Equation \eqref{ideal_A} using Equation \eqref{EQ2} we get
$$\Bigl(\rho(L_{(g')^{-1}})(A_{g'})\Bigr)\Bigl(A_{(g'')^{-1}}A_{g''}\Bigr) \subset A_1\bigl(A_{(g'')^{-1}}A_{g''}\bigr) \subset A_{(g'')^{-1}}A_{g''} \subset A_{[g],1}$$
Similarly, can be proven the third summand in Equation \eqref{ideal_A}.

Finally, for the first summand in \eqref{ideal_A}, as with the second summand, by Equation \eqref{EQ2} we have
$$\bigl(\rho(L_{(g')^{-1}}(A_{g'})\bigr)\bigl(\rho(L_{(g'')^{-1}})(A_{g''})\bigr) \subset A_1\bigl(\rho(L_{(g'')^{-1}})(A_{g''})\bigr) \subset A_{[g],1}.$$

\noindent That is, with Equation \eqref{ideal_A} we have shown
\begin{equation}\label{eq0.6}
A_{[g],1}A_{[g],1} \subset A_{[g],1}.
\end{equation}

\noindent From Equations \eqref{cero}-\eqref{eq0.5} and \eqref{eq0.6} we get $\mathcal{A}_{[g]}\mathcal{A}_{[g]} \subset \mathcal{A}_{[g]}.$
\end{proof}

\begin{proposition}\label{pro9}
For any $[g], [h] \in \Lambda_G/ \approx$ such that $[g] \neq [h]$ we have $\mathcal{A}_{[g]}\mathcal{A}_{[h]}=0$.
\end{proposition}
\begin{proof}
We have
\begin{equation}\label{cuatro}
\Bigl(A_{[g],1} \oplus A_{[g]}\Bigr)\Bigl(A_{[h],1} \oplus A_{[h]}\Bigr) \subset A_{[g],1}A_{[h],1} + A_{[g],1}A_{[h]} + A_{[g]}A_{[h],1} + A_{[g]}A_{[h]}.
\end{equation}

Consider the fourth summand $A_{[g]}A_{[h]}$ and suppose there exist $g' \in [g],$ $h' \in [h]$ such that $A_{g'}A_{h'} \neq 0$, so $A_{g'h'} \neq 0$. Observe that necessarily $h' \neq (g')^{-1}$, then $g'h' \in \Lambda_G$. By Remark \ref{Remark2} we obtain $g' \approx g'h'$, meaning that $g'h'\in [g]$. Similarly, $g'h' \in [h]$, so $[g] = [h]$, a contradiction. Hence $A_{g'}A_{h'} = 0$ and so
\begin{equation}\label{nueve}
A_{[g]}A_{[h]} = 0.
\end{equation}

Consider now the second summand $A_{[g],1}A_{[h]}$ in Equation \eqref{cuatro}. We take $g' \in [g] \cap \Sigma_G$ and $h' \in [h]$ such that $$\Bigl(\rho(L_{(g')^{-1}})(A_{g'}) + A_{(g')^{-1}}A_{g'}\Bigr)A_{h'} \neq 0.$$

Suppose $(A_{(g')^{-1}}A_{g'})A_{h'} \neq 0.$ By using associativity of $A$ we get $A_{(g')^{-1}}(A_{g'}A_{h'}) \neq 0,$ so $A_{g'h'} \neq 0$ and then $g'h' \in \Lambda_G$. Arguing as above $g \approx h$, a contradiction. If the another summand $\rho(L_{(g')^{-1}})(A_{g'})A_{h'} \neq 0$, since $\rho(L_{(g')^{-1}})$ is a derivation then $\rho(L_{(g')^{-1}})(A_{g'}A_{h'})$ or $A_{g'}\rho(L_{(g')^{-1}})(A_{h'})$ is non-zero, but in any case we argue similarly as above to get $g \approx h,$ a contradiction. From here
\begin{equation}\label{lex}
A_{[g],1}A_{[h]} = 0.
\end{equation}
By commutativity, for the third summand also $A_{[g]}A_{[h],1} = 0.$

Finally, let us prove $A_{[g],1}A_{[h],1} = 0$. Suppose there exist $g' \in [g] \cap \Sigma_G, h' \in [h] \cap \Sigma_G$ such that
$$\rho(L_{(g')^{-1}})(A_{g'})\rho(L_{(h')^{-1}})(A_{h'}) + \rho(L_{(g')^{-1}})(A_{g'})(A_{(h')^{-1}}A_{h'})$$
\begin{equation}\label{ideal_AA}
+ (A_{(g')^{-1}}A_{g'})\rho(L_{(h')^{-1}})(A_{h'}) + (A_{(g')^{-1}}A_{g'})(A_{(h')^{-1}}A_{h'}) \neq 0.
\end{equation}

If the last summand in Equation \eqref{ideal_AA} is non-zero, by the commutativity and associativity of $A$, since $h' \neq (g')^{-1}$, we have
$$(A_{(g')^{-1}}A_{g'})(A_{(h')^{-1}}A_{h'}) = (A_{(g')^{-1}}A_{(h')^{-1}})(A_{g'}A_{h'}) \subset A_{(g'h')^{-1}}A_{g'h'}$$
and by Remark \ref{Remark2} we get $g'h'\in [g]$ as well $g'h'\in [h]$, a contradiction.

If for the second summand in Equation \eqref{ideal_AA} there exist $g' \in [g]\cap \Sigma_G, h'\in [h]$, since $\rho(L_{(g')^{-1}})$ is a derivation we get
\begin{align*}
\Bigl(\rho(L_{(g')^{-1}})(A_{g'}&)\Bigr)\Bigl(A_{(h')^{-1}}A_{h'}\Bigr) \subset
\Bigl(\rho(L_{(g')^{-1}})(A_{g'}) A_{(h')^{-1}}\Bigr) A_{h'}\\
&\subset
 \rho(L_{(g')^{-1}})\Bigl(A_{g'}A_{(h')^{-1}}\Bigr)A_{h'} + \Bigl(A_{g'}\rho(L_{(g')^{-1}})(A_{(h')^{-1}})\Bigr) A_{h'}
\end{align*}
If the first summand or the second one is nonzero we get as in the previous cases that $[g]=[h]$, a contradiction.
Similarly, can be proven the third summand in Equation \eqref{ideal_AA}.

Finally, if for the first summand in \eqref{ideal_AA} there exist $g' \in [g]\cap \Sigma_G,$ $h'\in [h]\cap \Sigma_G$ we have
\begin{align*}
\Bigl(\rho(L_{(g')^{-1}})(A_{g'})\Bigr)&\Bigl(\rho(L_{(h')^{-1}})(A_{h'})\Bigr)\subset  \rho(L_{{(g')^{-1}}})\Bigl((A_{g'})\rho(L_{{(h')^{-1}}})(A_{h'})\Bigr)\\
&+ A_{g'}\rho(L_{{(g')^{-1}}})\Bigl(\rho(L_{{(h')^{-1}}})(A_{h'})\Bigr)\\
&\subset \rho(L_{{(g')^{-1}}})\Bigl(\rho(L_{{(h')^{-1}}})(A_{g'}A_{h'}) + \rho(L_{{(h')^{-1}}})(A_{g'})(A_{h'})\Bigr)\\
&+ A_{g'}\rho(L_{{(h')^{-1}}})\Bigl(\rho(L_{{(g')^{-1}}})(A_{h'})\Bigr) + A_{g'}\rho\Bigl([L_{{(h')^{-1}}},L_{{(g')^{-1}}}]\Bigr)A_{h'}
\end{align*}
and if some summand is non-zero arguing as above we obtain again the contradiction $[g]=[h]$.

Since Equation \eqref{ideal_AA} vanishes we assert
\begin{equation}\label{tex}
A_{[g],1}A_{[h],1} = 0,
\end{equation}
and from \eqref{cuatro}-\eqref{lex} and \eqref{tex} we conclude $\mathcal{A}_{[g]}\mathcal{A}_{[h]} = 0$.
\end{proof}

We recall that a subspace $I$ of a commutative and associative algebra $A$ is an {\it ideal of} $A$ if $AI \subset I$. In case $A$ is a $G$-graded algebra, we say an ideal $I \subset A$ is {\em graded} if it splits as $I = \oplus_{g \in G}I_g$, where $I_g := I \cap A_g$, and satisfies $I_gI_{g'} \subset I_{gg'}$, for $g,g'\in G$. We say that $A$ is {\it gr-simple} if $AA \neq 0$ and it contains no proper graded ideals.

\begin{theorem}\label{teo-11}
Let $A$ be a commutative and associative $\mathbb{F}$-algebra associated to a $G$-graded Lie-Rinehart $\mathbb{F}$-algebra $L.$ Then the following assertions hold.
\begin{enumerate}
\item[{\rm i)}] For any $[g] \in \Lambda_G/\approx$, the linear space $$\mathcal{A}_{[g]} = A_{[g],1} \oplus A_{[g]}$$ is a graded ideal of $A$.

\item[{\rm ii)}] If $A$ is gr-simple then all elements of $\Lambda_G$ are $\Lambda_G$-connected. Furthermore, $$A_1 = \sum\limits_{g \in \Lambda_G \cap \Sigma_G}\rho(L_{g^{-1}})(A_g) + \sum\limits_{g \in \Lambda_G}A_{g^{-1}}A_g$$
\end{enumerate}
\end{theorem}

\begin{proof}
i) Since $\mathcal{A}_{[g]}A_1 \subset \mathcal{A}_{[g]}$, Propositions \ref{pro2} and \ref{pro9} allow us to assert  $$\mathcal{A}_{[g]}A = \mathcal{A}_{[g]}\Bigl(A_1 \oplus (\bigoplus\limits_{g' \in [g]}A_{g'}) \oplus (\bigoplus\limits_{h \notin [g]}A_{h})\Bigr) \subset \mathcal{A}_{[g]}.$$ We conclude $\mathcal{A}_{[g]}$ is an ideal of $A$ and, since by construction is $G$-graded, is a graded ideal.

\medskip

ii) The gr-simplicity of $A$ implies $\mathcal{A}_{[g]} = A$, for $g \in \Lambda_G$. From here, it is clear that $[g] = \Lambda_G$ and $A_1 = \sum\limits_{g \in \Lambda_G \cap \Sigma_G} \rho(L_{g^{-1}})(A_g) + \sum\limits_{g \in \Lambda_G}A_{g^{-1}}A_g$.
\end{proof}

\begin{theorem}\label{teo2}
Let $A$ be a commutative and associative $\mathbb{F}$-algebra associated to a $G$-graded Lie-Rinehart $\mathbb{F}$-algebra $L.$ Then $$A = V + \sum\limits_{[g] \in \Lambda_G/\approx}\mathcal{A}_{[g]},$$ where $V$ is a linear complement in $A_1$ of $\sum_{g \in \Lambda_G \cap \Sigma_G}\rho(L_{g^{-1}})(A_g) + \sum_{g \in \Lambda_G} A_{g^{-1}}A_g$ and any $\mathcal{A}_{[g]}$ is one of the graded ideals of $A$ described in Theorem \ref{teo-11}-i). Furthermore, $\mathcal{A}_{[g]}\mathcal{A}_{[h]} = 0$ when $[g] \neq [h].$
\end{theorem}

\begin{proof}
We know that $\mathcal{A}_{[g]}$ is well-defined and, by Theorem \ref{teo-11}-i), a graded ideal of $A$, being clear that $$A = A_1 \oplus (\bigoplus\limits_{g \in \Lambda_G}A_g) = V + \sum\limits_{[g] \in \Lambda_G/\approx}\mathcal{A}_{[g]}.$$ Finally, Proposition \ref{pro9} gives $\mathcal{A}_{[g]}\mathcal{A}_{[h]}=0$ if $[g] \neq [h].$
\end{proof}

\medskip

Let us denote by ${\mathcal Ann}(A) := \{a \in A : aA = 0\}$ the {\it annihilator} of the commutative and associative algebra $A$.
% We recall that $A$ is called {\it perfect} if ${\mathcal Ann}(A) = 0$ and {\bf ES PRECISA ESTA SEGUNDA CONDICION? $AA = A$}.

\begin{corollary}\label{coro1}
Let $A$ be a commutative and associative $\mathbb{F}$-algebra associated to a $G$-graded Lie-Rinehart $\mathbb{F}$-algebra $L.$ If ${\mathcal Ann}(A) = 0$ and $$A_1 = \sum\limits_{g \in \Lambda_G \cap \Sigma_G}\rho(L_{g^{-1}})(A_g) + \sum\limits_{g \in \Lambda_G}A_{g^{-1}}A_g,$$ then $A$ is the direct sum of the  graded ideals given in Theorem \ref{teo-11}-i),
$$A = \bigoplus\limits_{[g] \in \Sigma_A/\approx}\mathcal{A}_{[g]}.$$ Furthermore, $\mathcal{A}_{[g]}\mathcal{A}_{[h]} = 0$ when $[g] \neq [h].$
\end{corollary}

\begin{proof}
Since $A_1 = \sum\limits_{g \in \Lambda_G \cap \Sigma_G}\rho(L_{g^{-1}})(A_g) + \sum\limits_{g \in \Lambda_G}A_{g^{-1}}A_g$ we obtain $A = \sum_{[g] \in \Lambda_G/\approx} \mathcal{A}_{[g]}$. To verify the direct character of the sum, take some $$a \in \mathcal{A}_{[g]} \cap \Bigl(\sum\limits_{[h]\in \Lambda_G/ \approx,
[h] \neq [g]}\mathcal{A}_{[h]}\Bigr).$$ Since $a \in \mathcal{A}_{[g]}$, the fact $\mathcal{A}_{[g]}\mathcal{A}_{[h]}=0$ when $[g] \neq [h]$ gives us  $$a\Bigl(\sum_{[h]\in \Lambda_G/ \approx, [h] \neq [g]} \mathcal{A}_{[h]}\Bigr) = 0.$$ In a similar way, since $a \in \sum_{[h]\in \Sigma_A/ \approx, [h] \neq [g]}\mathcal{A}_{[h]}$ we get $a\mathcal{A}_{[g]}=0.$ That is, $a \in {\mathcal Ann}(A)=0$.
\end{proof}

%%%%%%%%%%%%%%%%%%%%%%%%%%%%%%%%%%%%%%%%%%%%%%%%%%%%%%%%%%
%%%%%%%%%%%%%%%%%%%%%%%%%%%%%%%%%%%%%%%%%%%%%%%%%%%%%%%%%%
\section{Relating the decompositions of $L$ and $A$}
%%%%%%%%%%%%%%%%%%%%%%%%%%%%%%%%%%%%%%%%%%%%%%%%%%%%%%%%%%
%%%%%%%%%%%%%%%%%%%%%%%%%%%%%%%%%%%%%%%%%%%%%%%%%%%%%%%%%%

The aim of this section is to show that the decompositions of $L$ and $A$ as direct sum of ideals, given in Sections 2 and 3 respectively, are closely related.

Given a graded Lie-Rinehart algebra $(L,A)$ we call the {\em annihilator of $A$ in  $L$} (that is, using the structure of $L$ as $A$-module)  the set $${\mathcal Ann}_L(A):=\{v \in L : Av = 0\}.$$
Obviously, ${\mathcal Ann}_L(A)$ is an ideal of $L$. In fact, $A{\mathcal Ann}_L(A)= 0,$ also $[L,{\mathcal Ann}_L(A)] \subset {\mathcal Ann}_L(A)$ since $$A[L,{\mathcal Ann}_L(A)] = [L,A{\mathcal Ann}_L(A)] + \rho(L)(A){\mathcal Ann}_L(A) = 0,$$ and finally it verifies Equation \eqref{cond_ideal} $$A\bigl(\rho({\mathcal Ann}_L(A))(A)L\bigr) =
\bigl(A\rho({\mathcal Ann}_L(A))(A)\bigr)L =\rho(A{\mathcal Ann}_L(A))(A)L = \rho(0)(A)L = 0.$$

\begin{definition}\label{tight}\rm
A $G$-graded Lie-Rinehart algebra $(L,A)$ is {\it tight} if ${\mathcal Z}(L)={\mathcal Ann}_L(A)={\mathcal Ann}(A)=\{0\},$ $AA = A,$ $AL = L$ and
\begin{eqnarray*}
&& L_1 = \Bigl(\sum_{g \in \Sigma_G \cap \Lambda_G}A_{g^{-1}}L_g\Bigr) + \Bigl(\sum_{g \in \Sigma_G}[L_{g^{-1}}, L_g]\Bigr),\\
&& A_1 = \Bigl(\sum\limits_{g \in \Lambda_G \cap \Sigma_G}\rho(L_{g^{-1}})(A_g)\Bigr) + \Bigl(\sum\limits_{g \in \Lambda_G}A_{g^{-1}}A_g\Bigr).
\end{eqnarray*}
\end{definition}

\noindent If $(L,A)$ is tight then Corollary \ref{coro-1} and Corollary \ref{coro1} say that $$\hbox{$L = \bigoplus\limits_{[g] \in \Sigma_G/\sim} I_{[g]}$ \hspace{0.4cm} and \hspace{0.4cm} $A =\bigoplus\limits_{[g] \in \Lambda_G/\approx} \mathcal{A}_{[g]}$},$$ with any $I_{[g]}$ a graded ideal of $L$ verifying $[I_{[g]},I_{[h]}]=0$ if $[g] \neq [h],$ and any $\mathcal{A}_{[g]}$ a graded ideal of $A$ satisfying $\mathcal{A}_{[g]}\mathcal{A}_{[h]}=0$ if $[g] \neq [h]$.

\begin{proposition}
Let $(L,A)$ be a tight $G$-graded Lie-Rinehart algebra. Then for $[g] \in \Sigma_G/\sim$ there exists a unique $[h] \in \Lambda_G/\approx$ such that $\mathcal{A}_{[h]}I_{[g]} \neq 0$.
\end{proposition}

\begin{proof}
First we prove the existence. Given $[g] \in \Sigma_G/\sim,$ let us suppose that $AI_{[g]} = 0$. Since $I_{[g]}$ is a graded ideal it follows $$\hbox{$[I_{[g]},AL] = \Bigl[I_{[g]}, \bigoplus\limits_{h \in \Sigma_G/ \sim} AI_{[h]}\Bigr] = [I_{[g]},AI_{[g]}] = 0.$}$$ By hypothesis $AL = L,$ then $I_{[g]} \subset \mathcal{Z}(L) = \{0\},$ a contradiction. Since $A =\bigoplus_{[g] \in \Lambda_G/\approx} \mathcal{A}_{[g]},$ there exists $[h] \in \Lambda_G/\approx$ such that $\mathcal{A}_{[h]}I_{[g]} \neq 0$.

Now we prove that $[h]$ is unique. Suppose that $m$ is another element of $G$ which satisfies ${\mathcal A}_{[m]}I_{[g]} \neq 0$. From $\mathcal{A}_{[h]}I_{[g]} \neq 0$ and ${\mathcal A}_{[m]}I_{[g]} \neq 0$ we can take $h' \in {[h]}, m' \in {[m]}$ and $g',g'' \in {[g]}$ such that $A_{h'}L_{g'} \neq 0$ and $A_{m'}L_{g''} \neq 0$. Since $g',g'' \in {[g]}$, we can fix a $\Sigma_G$-connection $$\{g', g_2,\dots,g_n\} \subset \Sigma_G^{\pm} \cup \Lambda_G^{\pm}$$ from $g'$ to $g''$.

We have to distinguish four cases:

\begin{itemize}
\item If $h'g' \neq 1$ and $m'g'' \neq 1$, then $h'g'$, $m'g'' \in \Sigma_G$, and so $h' \approx m'.$ Indeed, in the case $g'g_2 \cdots g_n = g'',$ the $\Lambda_G$-connection from $h'$ to $m'$ is $$\{h',g', (h')^{-1},g_2,\dots,g_n,m',(g'')^{-1}\},$$ and in the case $g'g_2 \cdots g_n= (g'')^{-1}$ is $$\{h',g',(h')^{-1},g_2,\dots,g_n, (m')^{-1},g''\}.$$ From here $[h]=[m]$.

\item If $h'g' = 1$ and $m'g'' \neq 1$, we get $h'=(g')^{-1}$,
$m'g'' \in \Sigma_G$ and then $$\{(g')^{-1},g_2^{-1},\dots,g_n^{-1},(m')^{-1}, g''\}$$ is a $\Lambda_G$-connection from $h'$ to $m'$ in the case $g'g_1 \cdots g_n= g'',$ while $$\{(g')^{-1},g_2^{-1},\dots,g_n^{-1},m',(g'')^{-1}\}$$ is a $\Lambda_G$-connection in the case $g'g_1 \cdots g_n= (g'')^{-1}$. From here, $[h]=[m]$.

\item Suppose $h'g' \neq 1$ and $m'g''=1$. We can argue as in the previous item to get $[h]=[m]$.

\item Finally, we consider $h'g'= m'g'' = 1$. Hence $h'=(g')^{-1}, m'=(g'')^{-1}.$ Then
$$\{(g')^{-1},g_2^{-1},\dots,g_n^{-1}\}$$ is a $\Lambda_G$-connection between $h'$ and $m'$ which implies $[h]=[m]$.
\end{itemize}

We conclude $[h]$ is the unique element in $\Lambda_G/\approx$ such that $\mathcal{A}_{[h]}I_{[g]} \neq 0$ for the given $[g] \in \Sigma_G/\sim$.
\end{proof}

\noindent Observe that the above proposition shows that $I_{[g]}$ is an $\mathcal{A}_{[h]}$-module. Hence we can assert the following result.

\begin{theorem}\label{lema_final}
Let $(L,A)$ be a tight graded Lie-Rinehart $\mathbb{F}$-algebra over an associative and commutative $\mathbb{F}$-algebra $A$. Then $$\hbox{$L =\bigoplus\limits_{i\in I}I_i$ \hspace{0.4cm} and \hspace{0.4cm} $A = \bigoplus\limits_{j \in J}A_j$}$$ where any $I_i$ is a non-zero graded ideal of $L$ satisfying $[I_i,I_h]=0$ when $i \neq h$, and any $A_j$ is a non-zero graded ideal of $A$ such that $A_jA_k=0$ when $j \neq k$. Moreover, both decompositions satisfy that for any $r \in I$ there exists a unique $\tilde{r} \in J$ such that $$A_{\tilde{r}}I_r \neq 0.$$ Furthermore, any $I_r$ is a graded Lie-Rinehart algebra over $A_{\tilde{r}}$.
\end{theorem}

\medskip

%%%%%%%%%%%%%%%%%%%%%%%%%%%%%%%%%%%%%%%%%%%%%%%%%%%%%%%%%%%%%%
\section{Decompositions through the families of the  simple ideals}
%%%%%%%%%%%%%%%%%%%%%%%%%%%%%%%%%%%%%%%%%%%%%%%%%%%%%%%%%%%%%%

In this last section we are going to show that, under mild conditions, the decomposition of a graded Lie-Rinehart algebra $(L,A)$ given in Theorem \ref{lema_final} can be obtained by means of the families of the gr-simple ideals of $L$ and of the gr-simple ideals of $A$. At following we suppose that $\Sigma_G$ is {\em symmetrical}, meaning that, if $g \in \Sigma_G$ then $g^{-1} \in \Sigma_G,$ and also that $\Lambda_G$ is symmetrical in the same sense.

Next we introduce the concepts of $G$-multiplicativity and maximal length in the framework of graded Lie-Rinehart algebras in a similar way to the ones for other classes of graded Lie algebras and graded Leibniz algebras (see \cite{YoGradLie, YoLeibniz} for these notions and examples).

\begin{definition}\rm
We say that a $G$-graded Lie-Rinehart algebra $(L,A)$ is {\it
$G$-multiplicative} if for any $g,h \in \Sigma_G$ and any $k,j \in \Lambda_G$ the following conditions hold.
\begin{itemize}
\item If $gh \in \Sigma_G$ then $[L_g, L_h] \neq 0.$

\item If $kg \in \Sigma_G$ then $A_kL_g \neq 0.$

\item If $kj \in \Lambda_G$ then $A_kA_j \neq 0.$
\end{itemize}
\end{definition}

\begin{definition}\rm
A graded Lie-Rinehart algebra $(L,A)$ is {\it of maximal length} if $\dim L_g = \dim A_k=1$ for $g \in \Sigma_G, k\in \Lambda_G$.
\end{definition}

\noindent Observe that if $L,A$ are gr-simple algebras then $\mathcal{Z}(L) = {\mathcal Ann}_L(A) = {\mathcal Ann}(A) =\{0\}$. Also, as consequence of Theorem \ref{teo-1}-ii) and Theorem \ref{teo-11}-ii), we get that all of the elements in $\Sigma_G,\Lambda_G$ are $\Sigma_G$-connected and $\Lambda_G$-connected, respectively, and $$L_1 = \Bigl(\sum_{g \in \Sigma_G \cap \Lambda_G}A_{g^{-1}}L_g\Bigr) + \Bigl(\sum_{g \in \Sigma_G}[L_{g^{-1}},L_{g}]\Bigr),$$    $$A_1 = \Bigl(\sum_{k \in \Lambda_G \cap \Sigma_G}\rho(L_{k^{-1}})(A_k)\Bigr) + \Bigl(\sum_{k \in \Lambda_G}A_{k^{-1}}A_k\Bigr).$$ From here, the conditions for $(L,A)$ of being tight (see Definition \ref{tight}) together with the one of having $\Sigma_G, \Lambda_G$ all of their elements $\Sigma_G$-connected and $\Lambda_G$-connected, respectively, are necessary conditions to get a characterization of the gr-simplicity of $L$ and $A$. Actually, under the hypothesis of being $(L,A)$ of maximal length and $G$-multiplicative, these are also sufficient conditions as Theorem \ref{teo100} shows.

\begin{lemma}\label{lemma_I_zero}
Let $(L,A)$ be a tight $G$-graded Lie-Rinehart algebra of maximal length and $G$-multiplicative. If $I$ is a non-zero graded ideal of $L$ then $I \not\subset L_1$.
\end{lemma}

\begin{proof}
Suppose there exists a non-zero graded ideal $I$ of $L$
such that $I \subset L_1$. Then $[\oplus_{g \in \Sigma_G} L_g ,I] \subset (\oplus_{g \in \Sigma_G} L_g) \cap L_1 = 0$ and so
\begin{equation}\label{I_zero}
[L_g ,I] = 0 \hspace{0.2cm} {\rm for} \hspace{0.1cm} {\rm any} \hspace{0.1cm} g \in \Sigma_G.
\end{equation}

\noindent Also, $(\oplus_{k \in \Lambda_G}A_k)I \subset (\oplus_{k \in \Sigma_G}L_k) \cap L_1 = 0,$ then
\begin{equation}\label{I_zero2}
A_kI = 0 \hspace{0.2cm} {\rm for} \hspace{0.1cm} {\rm any} \hspace{0.1cm} k \in \Lambda_G.
\end{equation}

\noindent Since $A_1 = \Bigl(\sum_{k \in \Lambda_G \cap \Sigma_G}\rho(L_{k^{-1}})(A_k)\Bigr) + \Bigl(\sum_{k \in \Lambda_G}A_{k^{-1}}A_k\Bigr)$ we get $$A_1I \subset \sum_{k \in \Lambda_G \cap \Sigma_G}\rho(L_{k^{-1}})(A_k)I + \sum_{k \in \Lambda_G}(A_{k^{-1}}A_k)I.$$Now observe that,  in the first summand,   by Equations \eqref{fundamental}, \eqref{I_zero} and  \eqref{I_zero2} we get $$\rho(L_{k^{-1}})(A_k)I \subset [L_{k^{-1}},A_kI] + A_k[L_{k^{-1}},I] = 0,$$
and, in the second summand,  since $I$ is an $A$-module and by Equation \eqref{I_zero2} we have $$(A_{k^{-1}}A_k)I= A_{k^{-1}}(A_kI) = 0.$$
Hence $A_1I=0$, and with Equation \eqref{I_zero2} we conclude $I \subset {\mathcal Ann}_L(A)=\{0\}$ because $L$ is tight, a contradiction. Therefore $I\not\subset L_1$.
\end{proof}

\begin{proposition}\label{teo-3}
Let $(L,A)$ be a tight $G$-graded Lie-Rinehart algebra of maximal length, $G$-multiplicative and all the elements of $\Sigma_G$ being $\Sigma_G$-connected. Then either $L$ is gr-simple or $L = I \oplus I'$ where $I,I'$ are gr-simple ideals of $L$.
\end{proposition}

\begin{proof}
Let $I$ be a non-zero graded ideal of $L$. By Lemma \ref{lemma_I_zero} we can write $$I = (I \cap L_1) \oplus \Bigl(\bigoplus_{g \in \Sigma_G}(I \cap L_g)\Bigr)$$ with $I \cap L_g \neq 0$ for at least one $g \in \Sigma_G$. Let us denote by $I_g := I \cap L_g$ and by $\Sigma_I := \{g \in \Sigma_G : I_g \neq 0\}$. Then we can write $I = (I \cap L_1) \oplus (\bigoplus_{g \in \Sigma_I}I_g).$ Let us distinguish two cases.

{\bf Case a).} In the first case assume there exists $g \in \Sigma_I$ such that $g^{-1} \in \Sigma_I$. In this case the goal is to prove that $L$ is gr-simple. As $0 \neq I_g \subset I$ and we can assert by the maximal length of $(L,A)$ that
\begin{equation}\label{I-1}
L_g \subset I.
\end{equation}
Now, take some $h \in \Sigma_G \setminus \{g,g^{-1}\}$. Since $g$ is $\Sigma_G$-connected to $h$, we have a $\Sigma_G$-connection $\{g_1,\dots,g_n\} \subset \Sigma_G \cup \Lambda_G$ with $n \geq 2$, from $g$ to $h$ satisfying:

\medskip

$\{g_1 =g ,g_1g_2,g_1g_2g_3,\dots,g_1g_2\cdots g_{n-1}\} \subset \Sigma_G.$

$g_1g_2 \cdots g_n \in \{h,h^{-1}\}.$

\medskip

\noindent Taking into account $g_1 \in \Sigma_I,$ we have that if $g_2 \in \Lambda_G$ (resp., $g_2 \in \Sigma_G$), the $G$-multiplicativity and the maximal length of $L$ allow us to assert that $$0 \neq A_{g_2}L_{g_1} = L_{g_1g_2} \hspace{0.4cm} (\hbox{resp., } 0 \neq [L_{g_1}, L_{g_2}] = L_{g_1g_2}).$$ Since $0 \neq L_{g_1} \subset I$, as consequence of Equation \eqref{I-1}, we get in both cases that $$0 \neq L_{g_1g_2} \subset I.$$

\noindent A similar argument applied to $g_1g_2 \in \Sigma_G, g_3 \in \Sigma_G$ and $g_1g_2g_3 \in \Sigma_G$ gives us $0 \neq L_{g_1g_2g_3} \subset I.$ We can iterate this process with the connection $\{g_1,\dots,g_n\}$ to get $$0 \neq L_{g_1g_2\cdots g_n} \subset I.$$

\noindent Thus, we show that
\begin{equation}\label{either-1}
\mbox{for any } h \in \Sigma_G, \mbox{ we have that } 0 \neq L_m \subset I \mbox{ for some } m \in \{h,h^{-1}\}.
\end{equation}

\noindent Since $g^{-1} \in \Sigma_I$ then $\{g_1^{-1},\dots,g_n^{-1}\}$ is a $\Sigma_G$-connection from $g^{-1}$ to $h$ satisfying $$g_1^{-1}g_2^{-1} \cdots g_n^{-1} = m^{-1}.$$ By arguing as above we get,
\begin{equation}\label{I-III}
0 \neq L_{m^{-1}} \subset I
\end{equation}
and so $\Sigma_I = \Sigma_G.$ From the fact $L_1 = \bigl(\sum\limits_{g \in \Sigma_G \cap \Lambda_G}A_{g^{-1}}L_{g}\bigr) + \bigl(\sum\limits_{g \in \Sigma_G}[L_{g^{-1}},L_g]\bigr)$ we also have
\begin{equation}\label{equa-2}
L_1 \subset I.
\end{equation}
From \eqref{I-1}-\eqref{equa-2} we obtain $L \subset I,$ and so $L$ is gr-simple.

\medskip

{\bf Case b).} Suppose that for any $g \in \Sigma_I$ we have that $g^{-1} \notin \Sigma_I.$  In this case the goal  is to prove that  $L = I \oplus I'$ where $I,I'$ are gr-simple ideals of $L$. Observe that by arguing as in the previous case we can write
\begin{equation}\label{eq50}
\Sigma_G = \Sigma_I \; \dot{\cup} \; \Sigma_I^{-1}
\end{equation}
where $\Sigma_I^{-1} := \{g^{-1}: g \in \Sigma_I\},$
because if $g \in \Sigma_G$ satisfies $g \notin \Sigma_I$, by sentence \eqref{either-1} necessarily  $g^{-1} \in \Sigma_I$ and then $g\in \Sigma_I^{-1}$.

Let us denote by $$I' := \Bigl(\sum\limits_{g \in \Sigma_I^{-1},g^{-1} \in \Lambda_G}A_{g^{-1}}L_g\Bigr) \oplus \Bigl(\bigoplus\limits_{g \in \Sigma_I^{-1}}L_g\Bigr).$$ By construction, all homogeneous components of $I'$ are not contained in $I$. Our next aim is to show that $I'$ is a graded ideal of $L$. By construction $I'$ is $G$-graded, let us prove that $I'$ is a Lie ideal of $L.$ We have $$[L,I'] = \Bigl[L_1 \oplus (\bigoplus\limits_{h \in \Sigma_G}L_h), \Bigl(\sum\limits_{g \in \Sigma_I^{-1},g^{-1} \in \Lambda_G}A_{g^{-1}}L_g\Bigr) \oplus \Bigl(\bigoplus\limits_{g \in \Sigma_I^{-1}}L_g\Bigr)\Bigr] \subset$$

$$\Bigl[L_1, \sum\limits_{g \in \Sigma_I^{-1},g^{-1} \in \Lambda_G}A_{g^{-1}}L_g\Bigr] + 
\Bigl(\bigoplus\limits_{g \in \Sigma_I^{-1}}L_g\Bigr)$$

{\small \begin{equation}\label{eq40}
+ \Bigl[\bigoplus\limits_{h \in \Sigma_G}L_h, \Bigl(\sum\limits_{g \in \Sigma_I^{-1},g^{-1} \in \Lambda_G}A_{g^{-1}}L_g\Bigr)\Bigr] + \Bigl[\bigoplus\limits_{h \in \Sigma_G}L_h,\Bigl(\bigoplus\limits_{g \in \Sigma_I^{-1}}L_g\Bigr)\Bigr].
\end{equation}}

For the first summand in \eqref{eq40}, if there exist $g_0 \in \Sigma_I^{-1}, g_0^{-1} \in \Lambda_G$ such that $[L_1, A_{g_0^{-1}}L_{g_0}] \neq 0$, by Equation \eqref{fundamental} we obtain
\begin{align*}
[L_1, A_{g_0^{-1}}L_{g_0}] &= A_{g_0^{-1}}[L_1,L_{g_0}] + \rho(L_1)(A_{g_0^{-1}})L_{g_0}\\
&\subset A_{g_0^{-1}}L_{g_0} \subset I'.
\end{align*}

Consider the third summand in \eqref{eq40}. Suppose for some $h\in \Sigma_G, g \in \Sigma_I^{-1}, g^{-1}\in \Lambda_G$ we get $[L_h,A_{g^{-1}}L_g] \neq 0$. Then in case $h = g$,  clearly $[L_g,A_{g^{-1}}L_g] \subset L_g \subset I',$ and that in case $h = g^{-1}$, since $I$ is an ideal and the maximal length give us $$L_{g^{-1}} = [L_{g^{-1}},A_{g^{-1}}L_g] \subset I \cap I' = 0,$$ a contradiction with $g^{-1} \in \Sigma_I$. Suppose $h \notin \{g,g^{-1}\}$, then by Equation \eqref{fundamental} and maximal length either $A_{g^{-1}}[L_h,L_g] = L_h \neq 0$ or $\rho(L_h)(A_{g^{-1}})L_g = L_h \neq 0$. In both cases, since $g^{-1} \in \Sigma_I$, we have by $G$-multiplicativity that $L_{h^{-1}} \subset I,$ that is, $h^{-1} \in \Sigma_I$. From here $h \in \Sigma_I^{-1}$ and then $L_h \subset I'$. Therefore $$\Bigl[\bigoplus\limits_{h \in \Sigma_G}L_h,\sum\limits_{g \in \Sigma_I^{-1}, g^{-1} \in \Lambda_G}A_{g^{-1}}L_g\Bigr] \subset I'.$$

\noindent Finally, if we consider the fourth summand in \eqref{eq40} and certain $h\in \Sigma_G, g \in \Sigma_I^{-1}$ satisfies $[L_h,L_g] \neq 0$, we have $[L_h, L_g] = L_{hg}$. Suppose $h \neq g^{-1}.$ Since $g^{-1} \in \Sigma_I$, the $G$-multiplicativity gives us $[L_{g^{-1}},L_{h^{-1}}] = L_{g^{-1}h^{-1}} \subset I$. We have $g^{-1}h^{-1} = (hg)^{-1} \in \Sigma_I$. Hence $hg \in \Sigma_I^{-1}$ and then $L_{hg} \subset I'$. Consider $h = g^{-1} \in \Sigma_I,$ in case $[L_{g^{-1}},L_g] \neq 0$ we have $[L_{g^{-1}},L_g] \subset I$. Then $L_g = \bigl[[L_{g^{-1}},L_g],L_g\bigr] \subset I.$ From here $g,g^{-1} \in \Sigma_I,$ a contradiction with \eqref{eq50}. Thus  $\bigl[\bigoplus_{h \in \Sigma_G}L_h,\bigoplus_{g \in \Sigma_I^{-1}}L_g\bigr] \subset I'$. We prove that  $I'$ is a (graded) Lie ideal of $L$.

Let us check that $AI' \subset I'. $ We have $$AI' = \Bigl(A_1 \oplus \bigoplus\limits_{k \in \Lambda_G}A_k\Bigr)\Bigl(\Bigl(\sum\limits_{g \in \Sigma_I^{-1}, g^{-1} \in \Lambda_G}A_{g^{-1}}L_g\Bigr) \oplus \Bigl(\bigoplus\limits_{g \in \Sigma_I^{-1}}L_g\Bigr)\Bigr) \subset $$
\begin{equation}\label{eq400}
I' + \bigl(\bigoplus\limits_{k \in \Lambda_G}A_k\bigr)\Bigl(\sum\limits_{g \in \Sigma_I^{-1}, g^{-1} \in \Lambda_G}A_{g^{-1}}L_g\Bigr) + \bigl(\bigoplus\limits_{k \in \Lambda_G}A_k\bigr)\Bigl(\bigoplus\limits_{g \in \Sigma_I^{-1}}L_g\Bigr).
\end{equation}

Consider the third summand in \eqref{eq400} and suppose that $A_kL_g \neq 0$ for certain $k \in \Lambda_G, g \in \Sigma_I^{-1}.$ In case $kg \in \Sigma_I$ we get $kg \in \Sigma_G$, and then $(kg)^{-1}=k^{-1}g^{-1} \in \Sigma_G$. By the $G$-multiplicativity of $(L,A)$ that $L_{k^{-1}g^{-1}} = A_{k^{-1}}L_{g^{-1}}\neq 0.$ Now by the maximal length of $L$ and the fact $g^{-1} \in \Sigma_I$, we get $A_{k^{-1}}L_{g^{-1}} = L_{k^{-1}g^{-1}} \subset I.$ Therefore $k^{-1}g^{-1} = (kg)^{-1} \in \Sigma_I,$ a contradiction. Hence $kg \in \Sigma_I^{-1}$.

We can argue as above with the second summand in \eqref{eq400} so as to conclude that $AI' \subset I'$.

Finally we show that $\rho(I')(A)L \subset I'.$ By Equation \eqref{fundamental} and since $L$ is tight we conclude $$\rho(I')(A)L \subset [I',AL]+A[I',L] \subset [I',L]+AI' \subset I',$$
so $I'$ is a (graded) ideal of the graded Lie-Rinehart algebra $(L,A)$.

Now since $[I',I] = 0$ it follows $\sum\limits_{g \in \Sigma_G}[L_g,L_{g^{-1}}] = 0,$ so by hypothesis must be $$L_1 = \Bigl(\sum\limits_{g \in \Sigma_I, g^{-1} \in \Lambda_G}A_{g^{-1}}L_g\Bigr) \oplus \Bigl(\sum\limits_{g \in \Sigma_I^{-1}, g^{-1} \in \Lambda_G}A_{g^{-1}}L_g\Bigr).$$ Indeed, the sum is direct because if there exists $$0 \neq x \in \Bigl(\sum\limits_{g \in \Sigma_I, g^{-1} \in \Lambda_G}A_{g^{-1}}L_g\Bigr) \cap \Bigl(\sum\limits_{g \in \Sigma_I^{-1}, g^{-1} \in \Lambda_G}A_{g^{-1}}L_g\Bigr),$$ taking into account ${\mathcal Z}(L)=\{0\}$ and $L$ is graded, there exists $0 \neq v_{g'} \in L_{g'}, g' \in \Sigma_G$, such that $[x,v_{g'}] \neq 0$, being then $L_{g'} \subset I \cap I' = 0$, a contradiction. Hence $x = 0$ and the sum is direct. Taking into account the above observation and Equation (\ref{eq50}) we have $$L = I \oplus I'.$$

\noindent  Finally, we can proceed with $I$ and $I'$ as we did for $L$ in the first case of the proof to conclude that $I$ and $I'$ are gr-simple ideals, which completes the proof of the theorem.
\end{proof}

\noindent In a similar way to Proposition \ref{teo-3} we can prove the next result.

\begin{proposition}\label{teo-33}
Let $(L,A)$ be a tight $G$-graded Lie-Rinehart algebra of maximal length, $G$-multiplicative and   all the elements of $\Lambda_G$ being $\Lambda_G$-connected. Then either $A$ is gr-simple or $A = J \oplus J'$ where $J,J'$ are gr-simple ideals of $A$.
\end{proposition}

\noindent Finally, we can prove the following theorem.

\begin{theorem}\label{teo100}
Let $(L,A)$ be a tight $G$-graded Lie-Rinehart algebra of maximal length, $G$-multiplicative, with symmetrical $\Sigma_G, \Lambda_G$ in such a way that $\Sigma_G, \Lambda_G$ have all their elements $\Sigma_G$-connected and $\Lambda_G$-connected, respectively. Then $$\hbox{$L =\bigoplus\limits_{i\in I}I_i$ \hspace{0.4cm} and \hspace{0.4cm} $A = \bigoplus\limits_{j \in J}A_j$}$$ where any $I_i$ is a gr-simple ideal of $L$ having all of its subindices $\Sigma_G$-connected and such that $[I_i,I_h]=0$ for any $h \in I$ with $i \neq h$; and any $A_j$ is a gr-simple ideal of $A$ satisfying $A_jA_k=0$ for any $k \in J$ such that $j \neq k.$ Furthermore, for any $r \in I$ there exists a unique $\tilde{r} \in J$ such that $$A_{\tilde{r}}I_r \neq 0.$$ We also have that any $I_r$ is a $G$-graded Lie-Rinehart algebra over $A_{\tilde{r}}$.
\end{theorem}

\begin{proof}
Taking into account Theorem \ref{lema_final} we can write $$L = \bigoplus\limits_{[g] \in \Sigma_G/\sim}I_{[g]}$$ as the direct sum of the family of ideals $I_{[g]}$, being each $I_{[g]}$ a $G$-graded Lie-Rinehart algebra having $[g]$ as $G$-support. Also we can write $A$ as the direct sum of the ideals $$A = \bigoplus\limits_{[k] \in \Lambda_G / \approx }\mathcal{A}_{[k]}$$ in such a way that any $\mathcal{A}_{[k]}$ has $[k]$ as $G$-support, and that for any $I_{[g]}$ there exists a unique $\mathcal{A}_{[k]}$ satisfying $\mathcal{A}_{[k]} I_{[g]} \neq 0,$ and being $(I_{[g]}, \mathcal{A}_{[k]})$ a $G$-graded Lie-Rinehart algebra.

In order to apply Propositions \ref{teo-3} and \ref{teo-33} to each $(I_{[g]}, \mathcal{A}_{[k]})$, we previously have to observe that the $G$-multiplicativity of $(L,A)$, Proposition \ref{pro-9} and Theorem \ref{teo2} show that $[g]$ and $[k]$ have, respectively, all of their elements connected through $\Sigma_G$-connections and $\Lambda_G$-connections contained in $[g]$ and $[k]$, respectively. Any of the $(I_{[g]},\mathcal{A}_{[k]})$ is $G$-multiplicative as consequence of the $G$-multiplicativity of $(L,A)$. Clearly $(I_{[g]},\mathcal{A}_{[k]})$ is of maximal length and tight, last fact consequence of tightness of $(L,A)$, Propositions \ref{teo-3} and \ref{teo-33}. So we can apply Propositions \ref{teo-3} and \ref{teo-33} to each $(I_{[g]}, \mathcal{A}_{[k]})$ so as to conclude that any $I_{[g]}$ is either gr-simple or the direct sum of gr-simple ideals $I_{[g]} = V \oplus V'$; and that any $\mathcal{A}_{[k]}$ is either gr-simple or the direct sum of gr-simple ideals $\mathcal{A}_{[k]} = W \oplus W'$. From here, it is clear that by writing $I_i = V \oplus V'$ and $\mathcal{A}_j = W \oplus W'$ if $I_i$ or $\mathcal{A}_j$ are not, respectively, gr-simple, then Theorem \ref{lema_final} allows as to assert that the resulting decomposition satisfies the assertions of the theorem.
\end{proof}

\end{document}